\documentclass[12pt,reqno]{amsart}

\usepackage{times}
\usepackage[T1]{fontenc}
\usepackage{mathrsfs}
\usepackage{latexsym}
\usepackage[dvips]{graphics}
\usepackage{epsfig}
\usepackage{amsmath,amsfonts,amsthm,amssymb,amscd}
\input amssym.def
\input amssym.tex
\usepackage{color}
\usepackage{hyperref}

\newcommand\be{\begin{equation}}
\newcommand\ee{\end{equation}}
\newcommand\bea{\begin{eqnarray}}
\newcommand\eea{\end{eqnarray}}
\newcommand\bi{\begin{itemize}}
\newcommand\ei{\end{itemize}}
\newcommand\ben{\begin{enumerate}}
\newcommand\een{\end{enumerate}}

\newtheorem{thm}{Theorem}[section]

\newtheorem{cor}[thm]{Corollary}
\newtheorem{lem}[thm]{Lemma}
\newtheorem{prop}[thm]{Proposition}

\newcommand{\twocase}[5]{#1 \begin{cases} #2 & \text{#3}\\ #4
&\text{#5} \end{cases}   }
\newcommand{\threecase}[7]{#1 \begin{cases} #2 &
\text{#3}\\ #4 &\text{#5}\\ #6 &\text{#7} \end{cases}   }
\newcommand{\fourcase}[9]{#1 \begin{cases} #2 &
\text{#3}\\ #4 &\text{#5}\\ #6 &\text{#7}\\ #8 &\text{#9} \end{cases}   }

\newcommand{\Z}{\ensuremath{\mathbb{Z}}}

\newcommand{\N}{\mathbb{N}}

\numberwithin{equation}{section}

\begin{document}

\title{Coordinate sum and difference sets of $d$-dimensional modular hyperbolas}

\author{Amanda Bower}
\email{amandarg@umd.umich.edu}\address{Department of Mathematics
and Statistics, University of Michigan-Dearborn, Dearborn, MI 48128}

\author{Ron Evans}
\email{revans@math.ucsd.edu}\address{Department of Mathematics, University of California San Diego, La Jolla, CA  92093}

\author{Victor Luo}
\email{victor.d.luo@williams.edu}\address{Department of Mathematics and Statistics, Williams College, Williamstown, MA 01267}

\author{Steven J. Miller}
\email{sjm1@williams.edu, Steven.Miller.MC.96@aya.yale.edu} \address{Department of Mathematics and Statistics, Williams College, Williamstown, MA 01267}




\date{\today}

\thanks{We thank Mizan R. Khan for introducing us to this problem and for helpful discussions and comments on an earlier draft. The first and third named authors were supported by Williams College and NSF grant DMS0850577; the fourth named author was partially supported by NSF grant DMS0970067.}

\begin{abstract}
Many problems in additive number theory, such as Fermat's last theorem and
the twin prime conjecture, can be understood by examining sums or differences of a set
with itself. A finite set $A \subset \mathbb{Z}$ is considered
sum-dominant if $|A+A|>|A-A|$. If we consider all subsets of $\{0, 1, \dots, n-1\}$, as $n\to\infty$ it is natural to expect that almost all subsets should be difference-dominant, as addition is commutative but subtraction is not; however, Martin and O'Bryant in 2007 proved that a positive percentage are sum-dominant as $n\to\infty$.

This motivates the study of ``coordinate sum dominance''.
Given $V \subset (\Z/n\Z)^2$, we call
$S:=\{x+y: (x,y) \in V\}$ a coordinate sumset and $D:=\{x-y: (x,y) \in V\}$
a coordinate difference set, and we say
$V$ is coordinate sum dominant if $|S|>|D|$.
An arithmetically interesting choice of $V$ is
$\bar{H}_2(a;n)$, which is the reduction modulo $n$ of the modular hyperbola
$H_2(a;n) := \{(x,y): xy \equiv a \bmod n, 1 \le x,y < n\}$.
In 2009, Eichhorn, Khan, Stein, and
Yankov determined the sizes of $S$ and $D$ for $V=\bar{H}_2(1;n)$
and investigated conditions for coordinate sum dominance.
We extend their results to reduced $d$-dimensional modular
hyperbolas $\bar{H}_d(a;n)$ with $a$ coprime to $n$.
\end{abstract}

\maketitle

\tableofcontents

\section{Introduction}
Let $A \subset \N \cup \{0\}$. Two natural sets to study are
\bea
A + A & \ = \ & \{ x+y: x, y \in A \} \nonumber \\
A - A & \ = \ & \{ x-y: x, y \in A \}.
\eea
The former is called the sumset and the latter the difference set. Many problems in additive number theory can be understood in terms of sum and difference sets. For instance, the Goldbach conjecture says that the even numbers greater than 2 are a subset of $P+P$, where $P$ is the set of primes. The twin prime conjecture states that there are infinitely many ways to write 2 as a difference of primes (and thus if $P_N$ is the set of primes exceeding $N$, $P_N-P_N$ always contains 2). If we let $A_n$ be the set of positive $n$\textsuperscript{th} powers, then Fermat's Last Theorem says $(A_n + A_n) \cap A_n = \emptyset$ for all $n > 2$.

Let $|S|$ denote the cardinality of a set $S$. A set $A$ is sum dominant if $|A+A| > |A-A|$. We might expect that almost all sets are difference dominant since addition is commutative while subtraction is not. However, in 2007 Martin and O'Bryant \cite{MO} proved that a positive percentage of sets are sum dominant; i.e., if we look at all subsets of $\{0, 1, \dots, n-1\}$ then as $n\to\infty$ a positive percentage are sum dominant. One explanation is that choosing $A$ uniformly from $\{0, 1, \dots,  n-1\}$ is equivalent to taking each element from 0 to $n-1$ to be in $A$ with probability 1/2. By the Central Limit Theorem this implies that there are approximately $n/2$ elements in a typical $A$, yielding on the order of $n^2/4$ pairs whose sum must be one of $2n-1$ possible values. On average we thus have each possible value realized on the order of $n/8$ ways. It turns out most possible sums and differences are realized (the expected number of missing sums and differences are 10 and 6, respectively). Thus most sets are close to being balanced, and we just need a little assistance to push a set to being sum-dominant. This can be done by carefully controlling the fringes of $A$ (the elements near 0 and $n-1$). Such constructions are the basis of numerous
results in the field; see for example \cite{ILMZ,MO,MOS,MPR,Zh1,Zh2}.

This motivates the study of ``coordinate sum dominance'' on fringeless sets
such as $(\Z/n\Z)^2$.
Given $V \subset (\Z/n\Z)^2$, we call
$S:=\{x+y: (x,y) \in V\}$ a coordinate sumset and $D:=\{x-y: (x,y) \in V\}$
a coordinate difference set, and we say
$V$ is coordinate sum dominant if $|S|>|D|$.
An arithmetically interesting choice of $V$ is
$\bar{H}_2(a;n)$, which is the reduction modulo $n$ of the modular hyperbola
\be \label{modhyp}
H_2(a;n) := \{(x,y): xy \equiv a \bmod n, 1 \le x,y < n\},
\ee
where $(a,n)=1$.
Eichhorn, Khan, Stein, and Yankov \cite{EKSY} determined the
cardinalities of $S$ and $D$ for $V=\bar{H}_2(1;n)$
and investigated conditions for coordinate sum dominance.
See \cite{MV} for additional results on related problems
in other modular settings.



The modular hyperbolas in \eqref{modhyp} have very interesting structure,
as is evidenced in Figure \ref{fig:modhyp1}. See Figures 1 through 4 of
\cite{EKSY} for additional examples.


\begin{center}
\begin{figure}
\includegraphics[scale=.65]{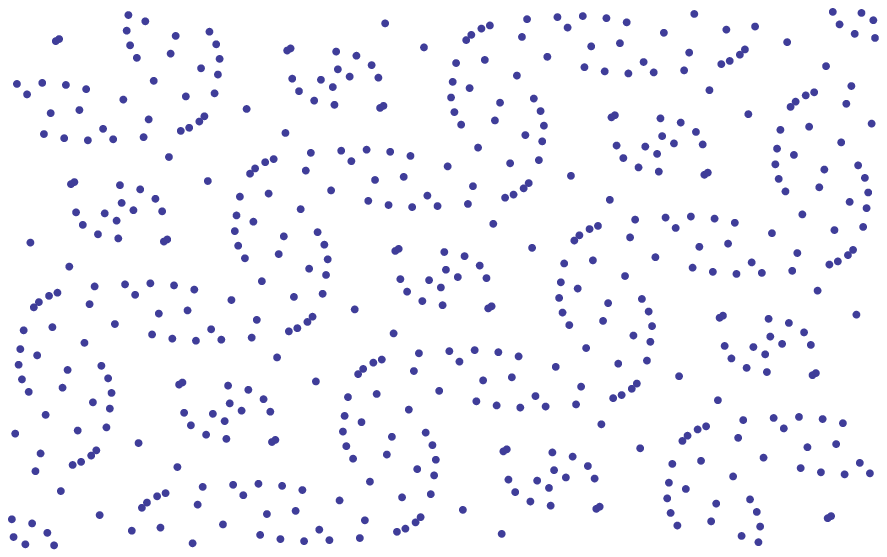}\ \ \ \ \ \ \ \ \ \
\includegraphics[scale=.65]{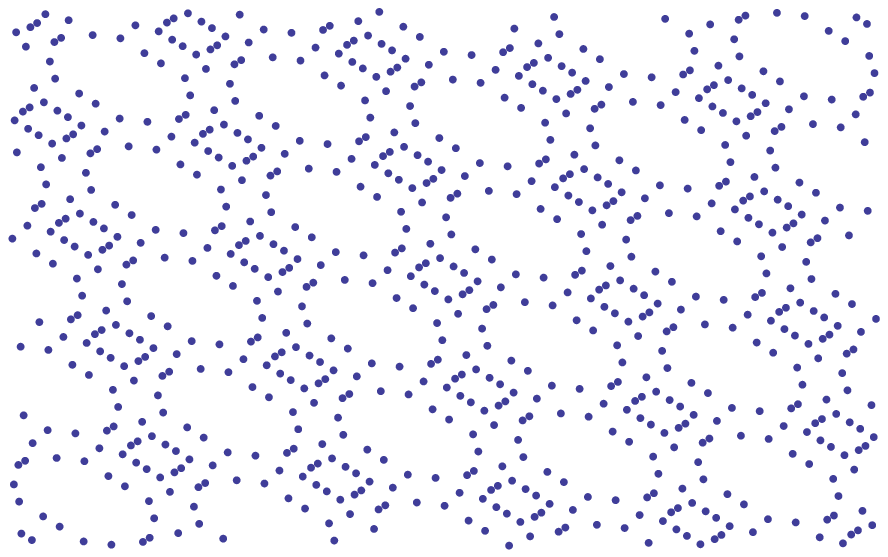}
\caption{\label{fig:modhyp1} (a) Left: $H_2(51; 2^{10})$.
(b) Right: $H_2(1325; 48^2)$.}
\end{figure}
\end{center}



In the sequel, coordinate sumsets will be the only type of sumset discussed.
Hence we may drop the premodifier ``coordinate'' without fear of confusion.
For $a$ relatively prime to $n$, we define the
sumset $S_2(a;n)$, the difference set $D_2(a;n)$, and their reduced
counterparts as
\bea
S_2(a;n) & \ = \ & \{x_1+x_2 : (x_1, x_2) \in H_2(a;n)\} \nonumber \\
D_2(a;n) & \ = \ & \{x_1- x_2: (x_1, x_2) \in H_2(a;n)\} \nonumber \\
\bar{S}_2(a;n) & \ = \ & \{x_1+x_2 \bmod n : (x_1, x_2) \in H_2(a;n)\} \nonumber \\
\bar{D}_2(a;n) & \ = \ & \{x_1-x_2 \bmod n : (x_1, x_2) \in H_2(a;n)\}.
\eea

From a geometric viewpoint,
$\#S_2(a;n)$ counts the number of lines of slope
$-1$ that intersect $H_2(a;n)$, and $\#D_2(a;n)$
counts the number of lines of slope 1 that intersect $H_2(a;n)$.
When the ratio
\be
c_2(a;n):=\# \bar{S}_2(a;n)/\#\bar{D}_2(a;n)
\ee
exceeds $1$, we have sum-dominance of $\bar{H}_2(a;n)$.

A $d$-dimensional modular hyperbola is of the form
\bea \label{highdimmodhyp}
 H_d(a;n) \ := \  \{(x_1, \dots, x_d): x_1 \cdots x_d \equiv
 a \bmod n, 1 \leq x_1,\dots, x_d <n\},
\eea
where $(a,n) = 1$.
We define the generalized signed sumset as
\bea
\bar{S}_d(m;a;n)=\{x_1 \ + \cdots + \ x_m - \cdots - x_d \bmod n:
(x_1, \dots, x_d) \in H_d(a;n) \},\ \ \
\eea
where $m$ is the number of plus signs in $\pm x_1\pm \cdots \pm x_d$.
In particular, $\bar{S}_2(1;a;n)=\bar{D}_2(a;n)$ and
$\bar{S}_2(2;a;n)=\bar{S}_2(a;n)$.


The goal of this paper is to extend results of \cite{EKSY} to
the general two-dimensional modular hyperbolas in \eqref{modhyp},
and to investigate the higher dimensional modular hyperbolas defined in
\eqref{highdimmodhyp}.
We prove explicit formulas for the cardinalities of the
sumsets $\bar{S}_2(a;n)$ and difference sets $\bar{D}_2(a;n)$
(Theorems \ref{thm:pis2indis2case} and \ref{p case}).  This
allows us to analyze the ratios $c_2(a;n)$
(Theorems \ref{3mod4} -- \ref{84}), thus providing conditions
on $a$ and $n$ for sum dominance and difference dominance of
reduced modular hyperbolas $\bar{H}_2(a;n)$.
For example, a special case of Theorem \ref{twoprimes}
shows that if $a=11$ and $n=3^t7^s$ with $t\ge 2$,
then $c_2(a;n)>1$, i.e., we have sum dominance.
A special case of
Theorem \ref{84} shows that when $a$ is a fixed power of 4,
we have sum dominance
for more than $85\%$ of those $n$ relatively prime to $a$.
For $d > 2$ and
positive integers $n$ whose prime factors all exceed 7, we prove in
Theorem \ref{higherdim} that $\#\bar{S}_d(m;a;n)=n$.  This means that
each such generalized sumset consists of all possible values mod $n$,
i.e., all possible sums and differences occur.


\section{Counting Preliminaries}

In this section we
present some counting results that are central to proving our main theorems. Many of these are natural generalizations of results from \cite{EKSY}, so we refer the reader to the appendix for detailed proofs.






Throughout this paper, $p$ always denotes a prime.
The following proposition reduces the analysis of the cardinalities of $\bar{S}_d(m;a;n)$ to those of  $\bar{S}_d(m;a;p^t)$, where $p^t$ is a factor in the canonical factorization of $n$.

\begin{prop}\label{mult}
Let $n=\prod_{i=1}^{k} p_i^{e_i}$ be the factorization of $n$ into distinct prime powers. Then
\bea
\#\bar{S}_d(m;a;n)\ = \ \prod_{i=1}^k \# \bar{S}_d(m;a;p_i^{e_i}).
\eea
\end{prop}

The proof is given in Appendix \ref{pf mult}.

Lemma \ref{SDrel} cuts our work in half, as once we understand the sumset we immediately have results for the corresponding difference set.


\begin{lem} \label{SDrel} We have $\bar{S}_2(a;n) \ = \ \bar{D}_2(-a;n)$.
\end{lem}
\begin{proof}
We show $\bar{S}_2(a;n) \subseteq \bar{D}_2(-a;n)$; the reverse containment is
handled similarly. Let $\tau \in \bar{S}_2(a;n)$.
Then there
exists $(x_0,y_0) \in H_2(a;n)$ such that $x_0y_0\equiv a \bmod n$ and
$ x_0+y_0 \equiv \tau \bmod n$.
Since $(x_0, n-y_0) \in H_2(-a; n)$ and $\tau \equiv x_0-(n-y_0) \bmod n$,
we see that $\tau \in  \bar{D}_2(-a;n)$.
\end{proof}



\begin{lem}\label{2k} We have $(2k \bmod p^t) \in \bar{D}_2(a;p^t) \Leftrightarrow (k^2+a)$ is a square modulo $p^t$. The map $f(k)=2k \bmod p^t$ defines a bijection \be f: \{k: k^2+a\ \text{\rm is a square} \bmod p^t, 0 \leq k < p^t\} \rightarrow \bar{D}_2(a;p^t) \ee when $p>2$. If $p=2$, then $f$ defines a bijection \be f: \{k: k^2+a\ \text{\rm is a square} \bmod 2^t, 0 \leq k < 2^{t-1}\} \rightarrow \bar{D}_2(a;2^t). \ee
%
%
\end{lem}
See Appendix \ref{pf 2k}. By Lemma 2.2, a similar result for $S$ in place of $D$ follows by replacing $a$ by $-a$.

\section{Cardinalities of $\bar{S}_2(a;p^t)$ and $\bar{D}_2(a;p^t)$}

In this section we compute the cardinalities of $\bar{S}_2(a;p^t)$ and $\bar{D}_2(a;p^t)$.
We then give conditions on $a$ and $n$  for sum dominance and
difference dominance of $\bar{H}_2(a;n)$.

\subsection{Case 1: $p = 2$}

We isolate a useful result that we need for the proof of the next lemma.

\begin{prop}[Gauss \cite{Ga}] \label{squaremod2}
For $t \ge 1$, any integer of the form $4^k(8n+1)$ is a square
$\bmod  2^t$.
\end{prop}

The next result is used in investigating some of the cases of Theorem \ref{thm:pis2indis2case}.

\begin{lem}\label{case3} Write $m=4b+r$ with $0 \le r \le 3$. For $t\geq 5$,
$k^2+3+8m$ is a square${}\bmod 2^t $ if and only if $ k\equiv \pm (4r+1) \bmod 16$.
\end{lem}

\begin{proof}
We only prove the case $r=0$, as the other proofs are similar.
First assume that $k^2+3+8m$ is a square${}\bmod 2^t$. Reducing${}\bmod 32$, we have that
$k^2+3+8m  \equiv k^2+3 \bmod 32$, which implies that $k\equiv \pm 1 \bmod 16$.
Conversely, assume that $k=16l \pm 1$ for some $l \in Z$.
Then $k^2+3+8m = (16l \pm 1)^2+3+8(4b)\equiv 256l^2 \pm 32l+4+32b
\equiv 4(8(8l^2 \pm l+b)+1) \bmod 2^t$.
Hence, by Proposition \ref{squaremod2}, $k^2+3+8m$ is a square${}\bmod 2^t$.
\end{proof}

\begin{thm}\label{thm:pis2indis2case}
For $t\geq 5$,
\bea
\threecase{\#\bar{D}_2(a;2^t) & \ = \ & }{ \frac{2^{t-4}}{3}+\frac{(-1)^{t-1}}{3}+3}{$a\equiv 7 \bmod 8$}{2^{t-3}}{$ a\equiv 1,5 \bmod \ 8$}{2^{t-4}}{$a\equiv 3 \bmod 8$} \nonumber \\
\threecase{\#\bar{S}_2(a;2^t) & \ = \ & }{ \frac{2^{t-4}}{3}+\frac{(-1)^{t-1}}{3}+3}{$a\equiv 1 \bmod 8$}{2^{t-3}}{$ a\equiv 3,7 \bmod \ 8$}{2^{t-4}}{$a\equiv 5 \bmod 8$.}
\eea
Moreover, $\#\bar{S}_2(a;16)=2\ $ for all $a$, and when $t \le 3$,
we have
$\#\bar{S}_2(a;2^t)=1$ with the exception that
$\#\bar{S}_2(a;8) = 2\ $ when $a \equiv 1 \bmod 4$.
\end{thm}

\begin{proof}
The claim for $t\le 4$ can be checked by direct calculation,
so assume $t\geq 5$.
By Lemma \ref{SDrel}, it is enough to prove the claims about the cardinality of $\bar{D}_2(a;2^t)$ when $a\equiv 1,3,5 \bmod 8$ and $\bar{S}_2(a;2^t)$ when $a\equiv 1 \bmod 8$.

We refer to the Appendix \ref{pf pis2indis2case} for the proofs of the results for $\bar{D}_2(a;2^t)$ when $a\equiv 1,5 \bmod 8$ and for $\bar{S}_2(a;2^t)$ when $a\equiv 1 \bmod 8$, since these proofs are straightforward
generalizations of proofs in \cite{EKSY}.


It remains to prove the result for the difference set when $a\equiv 3 \bmod 8$.
Write $a = 3 + 8m$.  We consider only the case where $m \equiv 0 \bmod 4$,
since the cases $m \equiv 1,2,3 \bmod 4$ are proved similarly and lead to
the same result.
By Lemma \ref{case3}, we see that if $m \equiv 0 \bmod 4$, then
\bea
& & \#\{k: k^2+3+8m \ \text{\rm is a square } \bmod 2^t, \ 0 \leq k < 2^{t-1}\}\nonumber\\ & &\ \ \ \ \  \ = \  \#\{1+16l: 0\leq l <2^{t-5}\} + \#\{15+16l: 0\leq l <2^{t-5}\} \ = \ 2^{t-4}.\ \ \ \ \
\eea
By Lemma \ref{2k}, we know
\be
\#\bar{D}(a;2^t) \ = \ \#\{k: k^2+3+8m\ \text{\rm is a square}
\bmod 2^t, \ 0 \leq k <2^{t-1}\}.
\ee
Hence $\#\bar{D}_2(a;2^t) \ = \ 2^{t-4}$.
\end{proof}


\subsection{Case 2: $p > 2$}

For this subsection, we adopt the following notation from \cite{EKSY}:
\bea
S_2'(a;p^t) & \ = \ & \{k \bmod p^t: k^2-a\text{  is a square} \bmod p^t, \ p \nmid (k^2-a)\} \nonumber \\
S_2''(a;p^t) & \ = \ & \{k \bmod p^t: k^2-a\text{ is a square} \bmod p^t, \ p | (k^2-a)\} 
\eea

By Lemma \ref{2k},
\bea
\#\bar{S}_2(a;p^t) & \ = \ & \#S_2'(a;p^t)+\#S_2''(a;p^t)
\eea

\begin{lem}\label{S_2'}
Let $p$ be an odd prime. Then
\bea
\twocase{\#\bar{S}_2'(a;p^t) \ = \ }{\frac{(p-1)p^{t-1}}{2}}{$\left(\frac{a}{p}\right)=-1$}{\frac{(p-3)p^{t-1}}{2}}{$\left(\frac{a}{p}\right)=1$.}
\eea
\end{lem}

See Appendix \ref{pf S_2'} for the proof.

\begin{lem}\label{S''_2} Let $p$ be an odd prime. If $\left(\frac{a}{p}\right)=-1$, then $\#S''_2(a;p^t)=0$ and thus $\#S_2(a;p^t)=\frac{\phi(p^t)}{2}$. If $\left( \frac{a}{p} \right) =1$, \bea
\#S''_2(a;p^t) \ = \ \frac{p^{t-1}}{p+1}+
\frac{3}{2}+\frac{(-1)^{t-1}(p-1)}{2(p+1)}.
\eea
\end{lem}

See Appendix \ref{pf S''_2} for the proof.

\begin{thm}\label{p case}
For $t\geq 1$ and $p > 2$,
\bea
\twocase{\#\bar{S}_2(a,p^t) & \ = \ & }{\frac{(p-3)p^{t-1}}{2}+\frac{p^{t-1}}{p+1}+\frac{3}{2}+\frac{(-1)^{t-1}(p-1)}{2(p+1)}}{$\left(\frac{a}{p}\right)=1$}{\frac{\phi(p^t)}{2}}{$\left(\frac{a}{p}\right)=-1$} \nonumber \\
\fourcase{\#\bar{D}_2(a,p^t) & \ = \ & }{\frac{(p-3)p^{t-1}}{2}+\frac{p^{t-1}}{p+1}+\frac{3}{2}+\frac{(-1)^{t-1}(p-1)}{2(p+1)}}{$p\equiv 1 \bmod 4, \left(\frac{a}{p}\right)=1$}{\frac{\phi(p^t)}{2}}{$p\equiv 1 \bmod 4,
\left(\frac{a}{p}\right)=-1$}{\frac{(p-3)p^{t-1}}{2}+\frac{p^{t-1}}{p+1}+\frac{3}{2}+\frac{(-1)^{t-1}(p-1)}{2(p+1)}}{$p\equiv 3 \bmod 4, \left(\frac{a}{p}\right)=-1$}{\frac{\phi(p^t)}{2}}{$p\equiv 3 \bmod 4, \left(\frac{a}{p}\right)=1$.} \nonumber\\
\eea
\end{thm}

\begin{proof}
The result follows from Lemmas \ref{S_2'}, \ref{S''_2}, and \ref{SDrel}.
\end{proof}

\begin{cor} \label{1mod4}
For $p\equiv 1\bmod 4$, $c_2(a;p^k)=1.$
\end{cor}


\subsection{Ratios for $d=2$}
Now that we have explicit formulas for the cardinalities of the
sum and difference sets, the next natural object to study is the ratio
$c_2(a;n)$ of the size of the sumset to the size of the difference set.
By Corollary \ref{1mod4}, we only need to consider the prime factors of $n$
which are congruent to $3 \bmod 4$, since the primes which are congruent
to $1 \bmod 4$ do not change $c_2(a;n)$.
When $p \equiv 3 \bmod 4$, it is sufficient to evaluate
$c_2(a;p^t)$ in the case when
$\left(\frac{a}{p} \right)=1$, since $c_2(-a;p^t)$ is the
reciprocal of $c_2(a;p^t)$.


%



%

\begin{thm}\label{3mod4}
For $p\equiv 3 \bmod 4$ and $\left(\frac{a}{p} \right)=1$,
\be
c_2(a;p^t)\ = \ 1 - 2\sum_{i=0}^{[t/2]-1} \frac{1}{p^{2i+1}} \ +
\frac{2}{\phi(p^t)}.
\ee
\end{thm}

\begin{proof}
By Theorem \ref{p case},
\bea
c_2(a;p^t) & \ = \ & \left( \frac{(p-3)p^{t-1}}{2}+\frac{p^{t-1}}{p+1}+
\frac{3}{2}+\frac{(-1)^{t-1}(p-1)}{2(p+1)} \right) \frac{1}{\phi(p^t)/2}
\nonumber \\
& = & \frac{p^2 -2p-1}{p^2-1} +
\frac{(-1)^{t-1}(p-1)+3p+3}{(p+1)\phi(p^t)}.
\eea
Therefore,
\bea
c_2(a;p^t) -1 - \frac{2}{\phi(p^t)} & \ = \ &
\frac{-2p}{p^2-1} + \frac{(-1)^{t-1}(p-1)+p+1}{(p+1)\phi(p^t)}
\nonumber \\
& = & \frac{-2p}{p^2-1} + 2\sum_{[t/2]}^{\infty} \frac{1}{p^{2i+1}}
= - 2\sum_{i=0}^{[t/2]-1} \frac{1}{p^{2i+1}}.
\eea
\end{proof}

\begin{thm} \label{twoprimes}
Let $p < q$ be primes, both congruent to $3 \bmod 4$, and let $s, t \ge 2$.
If $a$ is a square $\bmod \ p$, then $c_2(a;p^t q^s) < 1$, so we have
difference dominance.
If $a$ is not a square $\bmod \ p$, then $c_2(a;p^t q^s) > 1$, so we have
sum dominance.
\end{thm}
\begin{proof}
\noindent  It suffices to prove the first assertion,
for then the second will follow by taking
the reciprocal.  By Theorem \ref{3mod4}, $c_2(a;p^t) < 1$.
If $a$ is a square $\bmod \ q$, then also $c_2(a;q^s) < 1$, so that
$c_2(a;p^t q^s) = c_2(a;p^t)c_2(a;q^s) < 1$, as desired.
Finally, assume that $a$ is not a square $\bmod \ q$.  Then it remains to
show that $c_2(-a;q^s) > c_2(a;p^t)$.  By Theorem \ref{3mod4}, $c_2(a;p^t)$
is monotone decreasing in $t$.  Therefore it suffices to show that
$\lim_{s \to \infty} c_2(-a;q^s) > c_2(a;p^2)$.  This inequality is equivalent
to $1 -2q/(q^2-1) > 1 - (2p-4)/(p^2-p)$, so we must show that
$(p-2)/(p^2-p) > q/(q^2-1)$.  Since the right member is a decreasing
function of $q$, it suffices to prove this inequality when $q=p+4$, and this
is easily accomplished.
\end{proof}

It is not hard to show that the conclusion of Theorem \ref{twoprimes}
still holds in the case $s=1$, $t \ge 2$.  However, the inequalities
are reversed in the case $t=1$, $s \ge 1$.

The next three theorems are straightforward
generalizations of results from \cite{EKSY}, so we omit the proofs.

\begin{thm}\label{maxmin}
Let $N_k=\prod_{i=1}^k p_i$, where $p_i$ is the
$i$\textsuperscript{${\rm th}$} prime that is congruent to $3$ modulo $4$.
Fix a perfect square $a$ relatively prime to all of the $p_i$.  Then
\be
c_2(a;N_k)\ \asymp\ \log\log N_k,
\ee
and for any $t\geq 2$,
\be c_2(a;N_k^t)\ \asymp\ (\log\log N_k)^{-1}.
\ee
\end{thm}

\begin{thm}
Fix an integer $a$.  Let $n$ run through the positive integers
relatively prime to $a$.  Then
\begin{enumerate}
\item
\be
\frac{1}{\log \log n}\ \ll\ c_2(a;n)\ \ll\ \log \log n,
\ee
\item
\be
\lim_{n\to \infty} \sup c_2(a;n)=\infty \quad \mbox{and} \quad
\lim_{n\to \infty} \inf c_2(a;n)=0,
\ee
\item
\be
\lim_{n\to \infty} \sup \frac{\# S(a;n)}{\# D(a;n)}=\infty \quad \mbox{and}
\quad \lim_{n\to \infty} \inf \frac{\# S(a;n)}{\# D(a;n)}=0.
\ee
\end{enumerate}
\end{thm}

\begin{thm} \label{84}
For a fixed nonzero integer $a$, let
$E_a$ denote the set of positive integers $n$ relatively prime to $a$ such that
$\left(\frac{a}{p}\right)=1$ for every prime $p \equiv 3 \bmod 4$ dividing $n$.
Let $C_a(L) = \{n \in E_a: c_2(a;n) > L\}$.
Define $E_a(x) = \{n \in E_a: n \le x\}$ and
$C_a(L,x) = \{n \in C_a(L): n \le x\}$.
Then the lower density  of $C_a(L)$ in $E_a$, defined by
$\lim \inf \#C_a(L,x)/\#E_a(x)$, satisfies the inequality
\be \label{liminf}
\lim_{x\to \infty}\inf \frac{\#C_a(1,x)}{\#E_a(x)} \geq\
K_a \prod \left( 1-\frac{1}{p^2}\right),
\ee
where the product is over all primes $p \equiv 3 \bmod 4$ for which
$\left(\frac{a}{p}\right)=1$, and where
\bea
\fourcase{K_a & \ = \ & }{1}
{$a\equiv 0 \bmod 2$}{63/64}{$ a\equiv 1 \bmod \ 8$}{31/32}{$a\equiv 5
 \bmod 8$}{15/16}{$a \equiv 3 \bmod 4$.}
\eea
Furthermore, for any constant $L>0$, the
lower density of $C_a(L)$
in $E_a$ is positive.
\end{thm}

For example, if $a$ is an odd power of 2,
then the lower density in \eqref{liminf} exceeds 97\%.
Note that if the condition $\left(\frac{a}{p}\right)=1$
is replaced by $\left(\frac{a}{p}\right)=-1$ throughout
the statement of Theorem \ref{84}, then by Lemma \ref{SDrel},
\eqref{liminf} holds with the inequality $c_2(a;n) > 1$ replaced by
$c_2(a;n) < 1$.

\section{Cardinality of $\bar{S}_d(m;a;n)$ for $d>2$}

We now turn our attention to modular hyperbolas with higher dimension ($d >2)$.
Suppose that $p >7$ for every prime $p$ dividing $n$.  Then
Theorem \ref{higherdim} shows that the higher dimensional generalized
sumsets $\bar{S}_d(m;a;n)$ all have cardinality $n$.  In particular, this
cardinality is the same for every value of $m$, i.e., there is no dependence
on the number of plus and minus signs.

\begin{thm}\label{higherdim}
If the prime factors of $n$ all exceed 7, then
$\#\bar{S}_d(m;a;n)= n$.
\end{thm}

\begin{proof} Let $q = p^t$ for a prime $p > 7$.
By Proposition \ref{mult}, it suffices to prove that
$\#\bar{S}_d(m;a;q)= q$.  We will show that for every $a$ coprime to $q$ and
every $b \bmod q$, the system of congruences
\bea \label{system}
x_1 + \cdots + x_d  & \ \equiv \ & b \bmod q \nonumber \\
x_1 \cdots x_d \ & \equiv \ & a \bmod q
\eea
has a solution. This suffices, because $x_i$ could be
replaced by $q-x_i$ for any collection of subscripts $i$.
If \eqref{system} can always be solved for $d=3$,
then it can always be solved for any $d >3$, by setting $x_i=1$ for $i >3$.
Thus assume that $d=3$.

Solving \eqref{system} is equivalent to solving the congruence
$xy(b-x-y) \equiv a \bmod q$ for $x, y \in (\Z/q\Z)^\ast$.
Replacing $y$ by $y^{-1}$ and then multiplying by $y$, we see that
this is equivalent to solving
\be \label{quadpoly}
x^2 + x(y^{-1} - b) + ay \equiv 0 \bmod q.
\ee
The quadratic polynomial in $x$ in \eqref{quadpoly} has discriminant
\be \label{discrim}
(-4ay^3 +b^2y^2 -2by +1)/y^2.
\ee
Let $R(y) \in (\Z/p\Z)[y]$ denote the cubic polynomial in $y$ obtained by
reducing the numerator in \eqref{discrim} $\bmod p$.  To solve
\eqref{quadpoly}, it remains to show that there exists $y \in (\Z/p\Z)^\ast$
for which $R(y)$ is a non-zero square $\bmod \ p$; this is
because a non-zero square $\bmod \ p$ is also
a square $\bmod \ q$ (see p. 46 of \cite{IR}).

Suppose for the purpose of contradiction that no term in the sum
\be
\sum_{y=1}^{p-1} \left(\frac{R(y)}{p}\right)
\ee
is equal to 1.  Then since $R(y)$ has at most 3 zeros in $(\Z/p\Z)^\ast$,
we have
\be \label{evalS}
S: =\sum_{y=0}^{p-1} \left(\frac{R(y)}{p}\right) = w-p,
\ee
for some $w \in \{1,2,3,4\}$.

Let $D$ denote the discriminant of $R(y)$.  Then
$D \equiv 16a(b^3 -27a) \bmod p$, and so $D$ vanishes if and only if
$a \equiv (b/3)^3 \bmod  p$.  When $D$ vanishes, it follows that
$b \in (\Z/p\Z)^\ast$ and $y=3/(4b)$ is a simple zero of $R(y)$.  We
conclude that $R(y)$ cannot equal a constant times the square of a polynomial
in $(\Z/p\Z)[y]$.  Therefore (see equation (6.0.2) in \cite{BEW}) we can
apply Weil's bound to conclude that $|S| < 2\sqrt{p}$.  Together with
\eqref{evalS}, this yields
\be
p -  2\sqrt{p} < w \le 4,
\ee
which contradicts the fact that $p > 7$.
\end{proof}

We remark that the conditions $p >7$ cannot be weakened in
Theorem \ref{higherdim}.  For example, \eqref{quadpoly} has no
solution when
$p=q=2$,  $b=0$ and $a=1$; when
$p=q=3$ and $b=a=1$; when
$p=q=5$, $b=1$ and $a=2$; and when
$p=q=7$, $b=0$ and $a=3$.


\section{Conclusion and Future Research}
We generalized work of \cite{EKSY} on the modular hyperbola $H_2(1,n)$
by examining more general modular hyperbolas $H_d(a;n)$.
The two-dimensional case ($d=2$) provided interesting conditions on
$a$ and $n$ for sum dominance and difference dominance. On the other hand,
for higher dimensions ($d >2$), all possible sums and difference are
realized when the prime factors of $n$ all exceed 7.

The following are some topics for future and ongoing research:
\begin{enumerate}

\item We can study the cardinality of sumsets and difference sets of the intersection of modular hyperbolas with other modular objects such as lower dimensional modular hyperbolas and modular ellipses. See \cite{HK} for work on the cardinality of the intersection of modular circles and $H_2(1;n)$.

\item  Extend Theorem \ref{twoprimes} by estimating
$c_2(a;n)$ in cases where $n$ has more than two prime factors of
the form $4k+3$.

\item Extend Theorem \ref{higherdim} by finding the cardinality of the
generalized higher dimensional sumsets in cases where $(n, 210) > 1$.

\item In higher dimensions ($d>2$), nearly every sum and difference is realized
for $\bar{H}_d(a;n)$.  The situation becomes more interesting if we replace
$\bar{H}_d(a;n)$ by a random subset chosen according to some probability
distribution depending on $d$.  If $S$ and $D$ denote the corresponding
sumset and difference set, we can then compare the random variables
$\#S$ and  $\#D$.
\end{enumerate}


\appendix


\section{Additional Proofs}
The following proofs are a natural extension of the proofs given by \cite{EKSY}, and are included for completeness.

\subsection{Proof of Proposition \ref{mult}}\label{pf mult}

\begin{proof}[Proof of Proposition \ref{mult}]
Consider
\bea
g: \bar{S}_d(m;a;n) \longrightarrow \prod_{i=1}^{k}\bar{S}_d(m;a \bmod p_i^{e_i}; p_i^{e_i})
\eea  defined by
\bea
g(x)\ = \ (x \bmod p_1^{e_1}, \dots, x \bmod p_k^{e_k}).
\eea
We claim $g$ is a bijection.

To show $g$ is injective, suppose $g(x)=g(y)$. Then we have $(x \bmod p_1^{e_1}, \dots, x \bmod p_k^{e_k})=(y \bmod p_1^{e_1}, \dots, y \bmod p_k^{e_k}) \Rightarrow x \equiv y \bmod p_i^{e_i}$ for $i=1, \dots, k$. Thus, by the Chinese Remainder Theorem, we have that $x\equiv y \bmod n$. But $0\leq x,y<n \Rightarrow x=y$, so $g$ is injective.

To show $g$ is surjective, let $(\alpha_1, \dots, \alpha_k) \in \prod_{i=1}^k \bar{S}_d (m;a \bmod p_i^{e_i};p_i^{e_i})$. Then, for each $i \in \{1, \dots, k\}$, there exists $(x_{1_i}, \dots, x_{d_i}) \in H_d(a;p_i^{e_i})$ such that $x_{1_i} + \cdots + x_{m_i} - \dots - x_{d_i} \equiv \alpha_i \bmod p_i^{e_i}.$ By the Chinese Remainder Theorem, the system of congruences
\bea
 x_r \ \equiv\ x_{r_i} \bmod p_i^{e_i}
\eea
where $1 \leq r \leq d$ has a unique solution $x_1=x_1', \dots, x_d=x_d' \bmod n$. Since $x_{1_i} \cdots x_{d_i} \equiv a \bmod p_i^{e_i}$ for all $i\in\{1,\dots, k\}$, we have, again by the Chinese remainder theorem, $x_1' \cdots x_d' \equiv a \bmod n$. So $g(x_{1}' + \cdots + x_{m}' - \cdots - x_{d}' \bmod n)=(\alpha_1, \dots, \alpha_k)$, and hence $g$ is a bijection, which completes the proof.
\end{proof}

\subsection{Proof of Lemma \ref{2k}} \label{pf 2k}

Before proving Lemma \ref{2k}, we state a useful lemma that is a simple observation and immediate generalization of a result from \cite{EKSY}.

\begin{lem} \label{even}
Let $(x_0,y_0) \in H_2(a; p^t)$. Then $x_0-y_0 \equiv 2k_1 \bmod p^t$ and $x_0+y_0 \equiv 2k_2 \bmod p^t$ for some $k_1, k_2 \in \Z$.
\end{lem}
\begin{proof}
If $p=2$, then $x_0$ and $y_0$ are both odd since
they are coprime to $p^t$,
so their sum and difference are even. If $p\neq 2$, then $2^{-1}$ exists${}
\bmod p^t$, so $x_0-y_0 \equiv 2k_1 \bmod p^t$ and $x_0+y_0 \equiv 2k_2 \bmod
p^t$ both have solutions.
\end{proof}

\begin{proof}[Proof of Lemma \ref{2k}]
Since the proofs are similar, we only prove (1).

``$\Rightarrow$'' Let $(x_0,y_0) \in H_2(a; p^t)$. By Lemma \ref{even}, $x_0-y_0\equiv 2k \bmod p^t$ for some $k \in \Z$. Note $x_0-y_0 \in \bar{D}_2(a;p)$. Upon completing the square, we obtain $k^2+a \equiv (x_0-k)^2 \bmod p^t$. Thus $k^2+a$ is a square $\bmod \ p^t$.

``$\Leftarrow$'' Suppose $k^2+a$ is a square${}\bmod p^t$. Then there exists $c \in \Z$ such that $c^2-k^2 \equiv a \bmod p^t$. It follows that $(x_0,y_0)=( (c+k) \bmod p^t, (c-k) \bmod p^t) \in H_2(a;p^t)$ and $ x_0-y_0\equiv 2k \bmod p^t$. \\

Now we show that $d_{p^t,a}$ is a bijection. If $p >2$, then $(d_{p^t,a})^{-1}(x)=2^{-1}x \bmod p^t$ is the inverse of $d_{p^t,a}(x)$.

Now suppose $p=2$. We first show $d_{p^t,a}$ is injective. Choose $k \in \Z$ such that $ 0\leq k < 2^{t-1}$ and $k^2+a$ is a square${}\bmod 2^t$. Consider $k_1=(k-2^{t-1}) \bmod 2^t$. Then $k_1^2+a$ is also a square${}\bmod 2^t$ (because $k_1^2+a\equiv k^2+2(2^{t-1})-2^tk+a\equiv k^2+a \mod 2^{t-1}$). The congruence $2x\equiv 2k \mod 2^t$ has exactly two distinct solutions, which must be $k$ and $k_1$. Since either $k$ or $k_1$ is less than $2^{t-1}$, we conclude that $d_{2^t,a}$ is injective.

To show $d_{p^t,a}$ is surjective, let  $t\in \bar{D}_2(a;2^t)$. Then by Lemma \ref{even} there exists $(x_0,y_0) \in H_2(a;2^t)$ such that $x_0-y_0\equiv t\equiv 2k \mod 2^t$ for some $k \in \Z$. Hence, since $t=2k$, $d_{p^t,a}(k)=t$.
\end{proof}

\subsection{Proof of Theorem \ref{thm:pis2indis2case} cases}\label{pf pis2indis2case}

The next proposition is from \cite{IR} (see page 46). It gives us a quick way to count squares modulo prime powers.

\begin{prop}\label{solve}
For the congruence
\begin{center}
$x^2\equiv a \bmod p^t$
\end{center}
where $p$ is a prime and $a$ is an integer such that $p\nmid a$,
we have the following:
\begin{enumerate}
\item If $p\neq2$ and the congruence $x^2 \equiv a \bmod p$ is solvable,
then for every
$t\geq 1$
the congruence $x^2\equiv a \bmod p^t$ is solvable with precisely two
distinct solutions.
\item
If $p=2$ and the congruence $x^2 \equiv a \bmod 2^3$ is solvable,
then for every $t \geq 3$ the congruence $x^2\equiv a \bmod 2^t$ is
solvable with precisely four distinct solutions.
\end{enumerate}
\end{prop}

\begin{proof}[Proof of Theorem \ref{thm:pis2indis2case}] We break the analysis by the congruence class of $a$ modulo 8.\\

\noindent \underline{Case 1: Difference set for $a\equiv 1 \bmod 8$.} By Proposition \ref{2k}, we know
\be
\#\bar{D}(a;2^t) \ = \ \#\{k: k^2+1+8m\ \text{\rm is a square} \bmod 2^t, 0 \leq k <2^{t-1}\}.
\ee
We claim that
\be k^2+1+8m\ \text{\rm is a square} \bmod 2^t\ \Leftrightarrow\ k=4l\ \text{\rm for some}\ l \in \Z.
\ee

\noindent ``$\Rightarrow$'' Assume $k^2+1+8m$ is a square${}\bmod 2^t$. Then $k^2+1$ is a square${}\bmod 8$, which yields $k\equiv 0,4 \bmod 8$. Hence $k=4l$ for some $l \in \Z$. \\

\noindent ``$\Leftarrow$'' Assume $k=4l$ for some $ l \in \Z$. We want to show $k^2+1+8m=(4l)^2+1+8m$ is a square${}\bmod 2^t$. Reducing modulo 8 gives us $(4l)^2+1+8m\equiv 1 \bmod 8$, which is a square for all $l$. Hence, by the second part of Proposition \ref{solve}, $(4l)^2+1+8m$ is a square $\bmod \ 2^t$ for all $l$. Thus
\be
\{k: k^2+1+8m\ \text{is a square} \bmod 2^t, 0 \leq k < 2^{t-1}\} \ = \ \{4l: 0 \leq l < 2^{t-3}\}.
\ee

\ \\

\noindent \underline{Case 2: Difference set for $a \equiv 5 \bmod 8$.}
Using a similar argument from above, we show that
\be
k^2+5+8m \ \text{\rm is a square} \bmod 2^t\ \Leftrightarrow\ k=2+4l\ \text{\rm for some} l \in \Z.
\ee

\noindent ``$\Rightarrow$'' Assume $k^2+5+8m$ is a square ${}\bmod 2^t$. Then $k^2+5 +8m\equiv k^2+5$ is a square ${}\bmod 8$, which implies $k\equiv 2,6 \bmod 8$ or $k=2+4l$ for some $l \in \Z$. \\

\noindent ``$\Leftarrow$'' Assume $k=2+4l$ for some $ l \in \Z$. Reducing modulo 8 gives us $(2+4l)^2+5+8m\equiv 4+16l+16l^2+5+8m\equiv 1 \bmod 8$, which is a square for all $l$. Hence, by the second part of Proposition \ref{solve}, $(4l)^2+1+8m$ is a square${}\bmod 2^t$ for all $l$. Since $k=2+4l$ and $0 \leq k < 2^{t-1}$ we have $0 \leq l < 2^{t-3}$. Hence
\be
\{k: k^2+1+8m\ \text{\rm is a square} \bmod 2^t, 0 \leq k < 2^{t-1}\}=\{4l: 0 \leq l < 2^{t-3}\}.
\ee

\ \\

\noindent \underline{Case 3: Sum set for $a \equiv 1 \bmod 8$.}
Assume $a\equiv 1 \bmod 8$. Thus $a=1+8t$ for some $t\in \Z$. We show that
\be
\#\{k: k^2-1-8t\ \text{\rm is a square} \bmod 2^t, 0 \leq k < 2^{t-1}\}\ = \ 2\#\{k^2\bmod 2^{t-4}\}.
\ee
Let $j \in \{k: k^2-1-8t$ is a square${}\bmod 2^t, 0 \leq k < 2^{t-1}\}$. So $j^2-1-8t$ is a square${}\bmod 2^t$. Note that $j=2l+1$ because if not, we can reduce${}\bmod 4$ and notice that $-1$ is a square${}\bmod 4$, which is a contradiction. Also, $(l^2+l-2t)$ is a square${}\bmod 2^{t-2}$. Hence
\bea
& & \{k: k^2-1-8t\ \text{\rm is a square} \bmod 2^t, 0 \leq k < 2^{t-1}\} \nonumber\\
& & \ \ \ \ = \  \#\{l: l^2+l-2t\  \text{\rm is a square} \bmod 2^{t-2}, 0 \leq l < 2^{t-2}\}. \eea
As
\bea
& & \#\{l: l^2+l-2t\ \text{\rm is a square} \bmod 2^{t-2}, 0 \leq k < 2^{t-2}, l\ {\rm even}\}\nonumber\\
& & \ \ \ \ \ =\ \#\{l: l^2+l-2t\ \text{is a square} \bmod 2^{t-2}, 0 \leq l < 2^{t-2}, l\ {\rm odd}\},
\eea
we find
\bea
& & \#\{l: l^2+l-2t\ \text{\rm is a square} \bmod 2^{t-2}, 0 \leq l < 2^{t-2}\}\nonumber\\
& & \ \ \ \ \ = \ \# 2  \{l: l^2+l-2t\ \text{\rm is a square} \bmod 2^{t-2}, 0 \leq l < 2^{t-2}, l\ \text{\rm odd}\}. \eea
If $l$ is odd, then $l^{-1}$ exists. Note that if $l^2+l-2t\equiv c^2 \bmod 2^{t-2}$, then multiplying by $l^{-1}$ yields $1+l^{-1}+2l^{-2}t\equiv l^{-2}c$, which is also a square${}\bmod 2^{t-2}$. Hence
\be
\#\{l: 1+l^{-1}-2l^{-2}t\ \text{\rm is a square} \bmod 2^{t-2}, 0 \leq l < 2^{t-2}, l\ \text{\rm odd} \}.
\ee

Since $1+l^{-1}+2l^{-2}t$ is even and a square, $1+l^{-1}+2l^{-2}t\equiv 4m^2 \bmod 2^{t-2}$ for some $m, 0 \leq m < 2^{t-4}$. Therefore
\be
\#\{l: l^2+l-2t\ \text{\rm is a square} \bmod \ 2^{t-2}, 0 \leq l < 2^{t-1}, l\ {\rm odd} \} \ = \
\#\{k^2 \bmod 2^{t-4} \};
\ee the proof of this case is completed by Proposition \ref{Stangl}.
\end{proof}

\subsection{Proof of Lemma \ref{S_2'}}\label{pf S_2'}

\begin{proof}[Proof of Lemma \ref{S_2'}]
Suppose $l \in \#S_2'(a;p^t)$. Then $\left(\frac{l^2-a}{p}\right)=1$ and
\bea \label{S'}
\#S_2'(a;p^t) & \ = \ & \frac{1}{2} \sum\limits_{l=0, (l^2-a,p)=1}^{p^t-1}  \left(\left(\frac{l^2-a}{p}\right)+1\right) \nonumber \\
& = & \frac{1}{2} \sum\limits_{k=0}^{p^{t-1}} \sum\limits_{l=0, l^2 \neq a \bmod p}^{p-1} \left(\left(\frac{(l+kp)^2-a}{p}\right)+1\right) \nonumber  \\
& = & \left( \sum\limits_{l=0, l^2\neq a}^{p-1} \left(\left(\frac{l^2-a}{p}\right)+1\right) \right) \frac{p^{t-1}}{2} \nonumber \\
& = & \left(-1+ \sum\limits_{l=0, l^2 \neq a \bmod p}^{p-1} 1 \right) \frac{p^{t-1}}{2},
\eea as (see for example page 63 of \cite{IR})
\bea
\sum\limits_{i=0}^{p-1} \left(\frac{i^2-a}{p}\right)=-1.
\eea
Substituting
\bea
\twocase{\sum\limits_{l=0\atop l^2  \neq a \bmod p}^{p-1} = }{p}{$\left(\frac{a}{p}\right)=1$}{p-2}{$\left(\frac{a}{p}\right)=-1$}
\eea
into \eqref{S'} gives us the desired result.
\end{proof}

\subsection{Proof of Lemma \ref{S''_2}}\label{pf S''_2}

We use the following fact in our proof.

\begin{prop}[Stangl \cite{St}] \label{Stangl}
Let $p$ be an odd prime. Then
\begin{center}
$\#\{k^2 \bmod p^t\}=\frac{p^{t+1}}{2(p+1)}+(-1)^{t-1}
\frac{p-1}{4(p+1)}+\frac{3}{4}$.
\end{center}
When $p=2$, we have that
\begin{center}
$\#\{k^2 \bmod 2^t\}=\frac{2^{t-1}}{3}+\frac{(-1)^{t-1}}{6}+\frac{3}{2}$.
\end{center}
\end{prop}

\begin{proof}
As $\left(\frac{a}{p}\right)=-1$, the congruence $x^2-a\equiv 0 \bmod p$ has no solutions. Therefore $\#S''_2(a;p^t)=0$.

For the other claim, let $k\in S''_2(a;p^t)$ and assume $\left(\frac{a}{p}\right)=1$. Then $k^2-a \equiv p^2m^2 \bmod p^t$ for some $ m, 0\leq m < p^{t-2}$.
We have two cases.\\

\noindent \underline{Case 1:} Suppose $t\leq 2$. In this case, we have $m=0$, and since the congruence $x^2\equiv a \bmod p^t$ has exactly 2 solutions, we conclude that $\#S''_2(a;p^t)=2$.\\

\noindent \underline{Case 2:} Suppose $t \geq 3$. Consider
\begin{align}
g: S_2''(a;p^t) \rightarrow \{m^2 \bmod p^{t-2}\}
\end{align}
defined by $g(k)=m^2$. By Proposition \ref{solve}, for each $m^2 \in \{m^2 \bmod p^{t-2}\}$, the congruence $x^2 \equiv p^2m^2+a \bmod p^t$ is solvable with precisely two solutions. Hence $g$ is surjective and each element in the image of $g$ has exactly two distinct element in its pre-image. Therefore, since  $\#S''_2(a;p^t)=\#\{m^2 \bmod p^{t-2}\}$, by Theorem \ref{Stangl} we are done.
\end{proof}


\ \\

\end{document}